\documentclass{mfat}
\pagespan{283}{294}

\usepackage{mmap}
\usepackage[utf8]{inputenc}

\newtheorem{theorem}[subsection]{Theorem}
\newtheorem{lemma}[subsection]{Lemma}

\theoremstyle{definition}
\newtheorem{definition}[subsection]{Definition}

\theoremstyle{remark}
\newtheorem{remark}[subsection]{Remark}
\newtheorem{example}[subsection]{Example}

\makeatletter \@addtoreset{subsection}{section}
\@addtoreset{equation}{section} \@addtoreset{figure}{section}
\@addtoreset{table}{section} \makeatother
\renewcommand{\theequation}{\thesection.\arabic{equation}}
\renewcommand{\thefigure}{\thesection.\arabic{figure}}
\renewcommand{\thetable}{\thesection.\arabic{table}}

\newcommand{\rr}{\mathbb{R}}
\newcommand{\nn}{\mathbb{N}}
\newcommand{\zz}{\mathbb{Z}}

\newcommand\id{\mathrm{id}}

\newcommand{\Hplus}{H^{+}(F)}

\newcommand{\HifixS[1]}{H^{+}(F,#1)}

\renewcommand{\Im}{\mathop{\mathrm{Im}}\nolimits}

\newcommand\ntN{\mathbb{N}}
\newcommand\ntNInv{-\mathbb{N}}
\newcommand\ntZ{\mathbb{Z}}

\newcommand\typeGen[1]{A_{#1}}

\newcommand\typeFinite[1]{\typeGen{[#1]}}
\newcommand\typeNatural{\typeGen{\ntN}}
\newcommand\typeNaturalInv{\typeGen{\ntNInv}}
\newcommand\typeIntegers{\typeGen{\ntZ}}

\newcommand\iind{i}
\newcommand\Iind{\Delta}
\newcommand\diam{\mathrm{diam}\,}

\newcommand\classHomeotopy{\EuScript{P}}

\newcommand\HFdSi[1]{H^{+}(F_{#1}, \partial_{-}S_{#1})}

\newcommand\HF{H(F)}

\newcommand\Xsp{X}
\newcommand\HFX{H^{+}(F_{\Xsp},\partial_{-}\Xsp)}

\begin{document}

\title[Homeotopy groups of rooted tree like nonsingular foliations]
         {Homeotopy groups of rooted tree like non-singular foliations on
         the plane}

\author{Yu. Yu. Soroka}
\address{Taras Shevchenko National University of Kyiv, Kyiv, Ukraine}
%\curraddr{}
\email{sorokayulya15@gmail.com}
%\thanks{Thank you.}

%    General info
%\subjclass[2000]{Primary ; Secondary }
\subjclass[2000]{Primary 57S05; Secondary 57R30, 55Q05 }
\date{31/03/2016}
%\dedicatory{This paper is dedicated to you.}
\keywords{Non-singular foliations, homeotopy groups.}

\begin{abstract}
  Let $F$ be a non-singular foliation on the plane with all leaves
  being closed subsets, $H^{+}(F)$ be the group of homeomorphisms of
  the plane which maps leaves onto leaves endowed with compact open
  topology, and $H^{+}_{0}(F)$ be the identity path component of
  $H^{+}(F)$. The quotient $\pi_0 H^{+}(F) = H^{+}(F)/H^{+}_{0}(F)$ is
  an analogue of a mapping class group for foliated homeomorphisms. We
  will describe the algebraic structure of $\pi_0 H^{+}(F)$ under
  an assumption that the corresponding space of leaves of $F$ has a
  structure similar to a rooted tree of finite diameter.
\end{abstract}

\maketitle

\section{Introduction}
Non-singular foliations on the plane were studied by
W.~Kaplan~\cite{Kaplan:DJM:1940,Kaplan:DJM:1941} and
H.~Whitney~\cite{Whitney:AM:1933} in the 40--50 years of the XX
century.  In particular, W.~Kaplan in~\cite{Kaplan:DJM:1941}
has generalized a theorem of E.~Kamke and proved that every non-singular
foliation $F$ on the plane admits a continuous function $f: \rr^{2}
\rightarrow \rr$ such that
\begin{enumerate}
\item [1)] the leaves of $f$ are connected components of level sets $f^{-1}(c)$, $c\in\rr$;
\item [2)] near each $z\in\rr^{2}$ there are local coordinates $(u,v)$ in which $f(u,v) = u + f(z)$.
\end{enumerate}
This result was further extended to foliations with singularities by W.~Boothby~\cite{Boothby:AJM_1:1951}, and J.~Jenkins and M.\,Morse~\cite{JenkinsMorse:AJM:1952}.
Topological classification of different kinds of functions on surfaces was investigated in many papers,
see e.g.\! A.~Fomenko and A.~Bolsinov~\cite{BolsinovFomenko:1997}, A.~Oshem\-kov~\cite{Oshemkov:PSIM:1995},
V.~Sharko~\cite{Sharko:UMZ:2003}, \cite{Sharko:Zb:2006}, O.~Prishlyak~\cite{Prishlyak:UMZ:2000},
\cite{Prishlyak:MFAT:2002}, E.~Polulyakh and I.~Yurchuk~\cite{PolulyakhYurchuk:ProcIM:2009},
E.~Polulyakh~\cite{Polulyakh:UMZ:ENG:2015}, V.~Sharko and Yu.~Soroka~\cite{SharkoSoroka:MFAT:2015}.

W.\,Kaplan in~\cite{Kaplan:DJM:1940,Kaplan:DJM:1941} has also
mentioned that a non-singular foliation on the plane is glued of
countably many strips along open boundary intervals and such that each
strip has a foliation by parallel lines.  In a recent paper
S.~Maksymenko and E.~Polulyah~\cite{MaksymenkoPolulyakh:PGC:2015}
studied non-singular foliations $F$ on arbitrary non-compact surfaces
$\Sigma$ glued from strips in a similar way.  They proved
contractibility of the connected components of groups $\HF$ of
homeomorphisms of $\Sigma$ mapping leaves onto leaves.  Thus the
homotopy type of $\HF$ is determined by the quotient group $\pi_0 \HF
= \HF/H_0(F)$ of path components of $\HF$, where $H_0(F)$ is the
identity path component of $\HF$.

In the present paper we compute the groups $\pi_0 \HF$ for a special
class of non-singular foliations on the plane whose spaces of leaves
have the structure similar to rooted trees of finite diameter, see
Theorem~\ref{th:main_theorem}.

\section{Striped surfaces}

Let $\Sigma_i$ be a surface with a foliation $F_i$, $i=1,2$.
Then a homeomorphism $h:\Sigma_1 \to \Sigma_2$ will be called
\emph{foliated} if it maps leaves of $F_1$
onto leaves of $F_2$.

\begin{definition}\label{df:model_strip} A subset $S \subset \rr^2$
    will be called a \emph{model strip} if the following two
    conditions hold:
\begin{enumerate}
\item[$1)$] $\rr \times (-1, 1) \ \subseteq \ S \ \subset \ \rr \times [-1,1]$;
\item[$2)$] $S \ \cap \ \rr \times \lbrace -1, 1 \rbrace$ is a union of open mutually disjoint finite intervals.
\end{enumerate}
\end{definition}

\medskip

Put
\iffalse
\begin{align*}
\partial_{-} S &= S \cap (\rr \times \lbrace -1 \rbrace), &
\partial_{+} S &= S \cap (\rr \times \lbrace 1 \rbrace), &
\partial S &= \partial_{-} S \cup \partial_{+} S.
\end{align*}
\fi
%%%
$$
\partial_{-} S = S \cap (\rr \times \lbrace -1 \rbrace), \quad
\partial_{+} S = S \cap (\rr \times \lbrace 1 \rbrace), \quad
\partial S = \partial_{-} S \cup \partial_{+} S.
$$
%%%

Notice that every model strip has an oriented foliation consisting of horizontal arcs  $\rr \times t$, $t \in (-1,1)$, and connected components of $\partial S$.

Let $\lbrace S_{\lambda} \rbrace _{\lambda \in \Lambda}$ be an arbitrary family of model strips, and
\iffalse
\begin{align*}
X &= \underset {\lambda \in \Lambda} {\cup} \partial_{-} S_{\lambda}, &
Y &= \underset {\lambda \in \Lambda} {\cup} \partial_{+} S_{\lambda}.
\end{align*}
\fi
$$
X = \underset {\lambda \in \Lambda} {\cup} \partial_{-}
S_{\lambda},\quad
Y = \underset {\lambda \in \Lambda} {\cup} \partial_{+} S_{\lambda}.
$$

By Definition~\ref{df:model_strip}, $X$ and $Y$ are disjoint unions of open intervals, therefore one can also
write
\iffalse
\begin{align*}
X&=\underset {\alpha \in A} {\cup} X_{\alpha}, &
Y&=\underset {\beta \in B} {\cup} Y_{\beta},
\end{align*}
\fi
$$
X=\underset {\alpha \in A} {\cup} X_{\alpha}, \quad
Y=\underset {\beta \in B} {\cup} Y_{\beta},
$$
where $X_{\alpha}$ and $Y_{\beta}$ are open boundary intervals of those models strips and $A$ and $B$ are some index sets.

We will now glue model strips $S_{\lambda}$ by identifying some of the intervals of $X_{\alpha}$ with some of the intervals of $Y_{\beta}$.
In order to make this let us fix any set of indexes $C$ and two injective maps $p: C \rightarrow A$ and $q: C \rightarrow B$.
Notice that for each $\gamma \in C$ there exists a unique preserving orientation affine homeomorphism $\varphi_{\gamma}: X_{p(\gamma)} \rightarrow Y_{q(\gamma)}$.
Then the quotient space
\begin{equation}\label{equ:striped_surface}
\Sigma: = \bigsqcup \limits_{\lambda\in \Lambda} S_{\lambda} \Bigr/ \{X_{p(\gamma)} \stackrel{\varphi_{\gamma}}{\sim} Y_{q(\gamma)} \}
\end{equation}
will be called a \emph{striped surface}.

\begin{remark}\rm
A unique preserving orientation affine homeomorphism $\phi:(a,b)\to(c,d)$ is given by $\phi(t) = \frac{c-d}{b-a}(t-a)$.
\end{remark}

\begin{remark}\label{rem:wider_class}\rm
  In \cite{MaksymenkoPolulyakh:PGC:2015} a wider class of striped
  surfaces is considered: it is also allowed to identify arbitrary
  connected components of $\bigsqcup \limits_{\lambda\in
    \Lambda}\partial S_{\lambda}$ and the gluing affine homeomorphisms
  may reverse orientation.
\end{remark}

Let also $p: \bigsqcup \limits_{\lambda\in \Lambda} S_{\lambda} \rightarrow \Sigma$ be the quotient map and $p_{\lambda}: S_{\lambda} \rightarrow \Sigma$ be the restriction of $p$ to the model strip $S_{\lambda}$.
Then the pair $(S_{\lambda}, p_{\lambda})$ will be called a \emph{chart} for the strip $S_{\lambda}$.

Since the homeomorphism $\varphi_{\gamma}$ identifies leaves of such foliations, we see that every striped
surface has the foliation $F$ consisting of foliations on model strips.
This foliation will be called \emph{canonical}.

Moreover, each leaf of the foliation on the model strip is oriented and the gluing preserves orientation.
Therefore all leaves of $F$ are oriented.

\medskip

{\bf Special leaves.}
Let $U\subset\Sigma$ be a subset.
Then the union of all leaves of $F$ intersecting $U$ is called the \emph{saturation} of $U$ with respect to $F$
and denoted by $Sat(U)$.

A leaf $\omega$ of $F$ will be called \emph{special} if
\[
\omega \neq \underset{N(\omega)} {\bigcap} \overline{Sat \left( N(\omega) \right)},
\]
where $N(\omega)$ runs over all open neighborhoods of $\omega$.

For instance each leaf $\omega$ belonging to the interior of a strip is non-special.
Moreover, suppose $\omega = X_{p(\gamma)} \sim Y_{q(\gamma)}$ is a leaf such that $\partial_{-} S_{\lambda} = X_{p(\gamma)}$ and $\partial_{+} S_{\lambda'} = Y_{p(\gamma)}$, see Figure~\ref{fig:nonspec_leaf}(a).
Then the topological structure of the foliation $F$ near $\omega$ is ``similar'' to the structure of $F$ near ``internal'' leaves of strips and such a leaf is non-special as well, see~\cite[Lemma~3.2]{MaksymenkoPolulyakh:PGC:2015}.

It also follows from that lemma that $\omega$ is special if and only if one of the following two conditions hold, see Figure~\ref{fig:nonspec_leaf}(b):
\begin{enumerate}
\item [1)] $\omega$ is the image of gluing of leaves $X_{p(\gamma)}$ and $Y_{q(\gamma)}$ such that either $X_{p(\gamma)} \subsetneq \partial_{-} S_{\lambda}$ or $ Y_{p(\gamma)} \subsetneq  \partial_{+} S_{\lambda'}$ for some $\gamma \in C$, $\lambda,\lambda'\in\Lambda$;
\item [2)] $\omega \subsetneq \partial_{-} S_{\lambda}$ or $\omega \subsetneq \partial_{+} S_{\lambda}$ for some $\lambda\in\Lambda$.
\end{enumerate}

\begin{figure}[h]
\begin{tabular}{ccc}
\includegraphics[height = 1.9 cm]{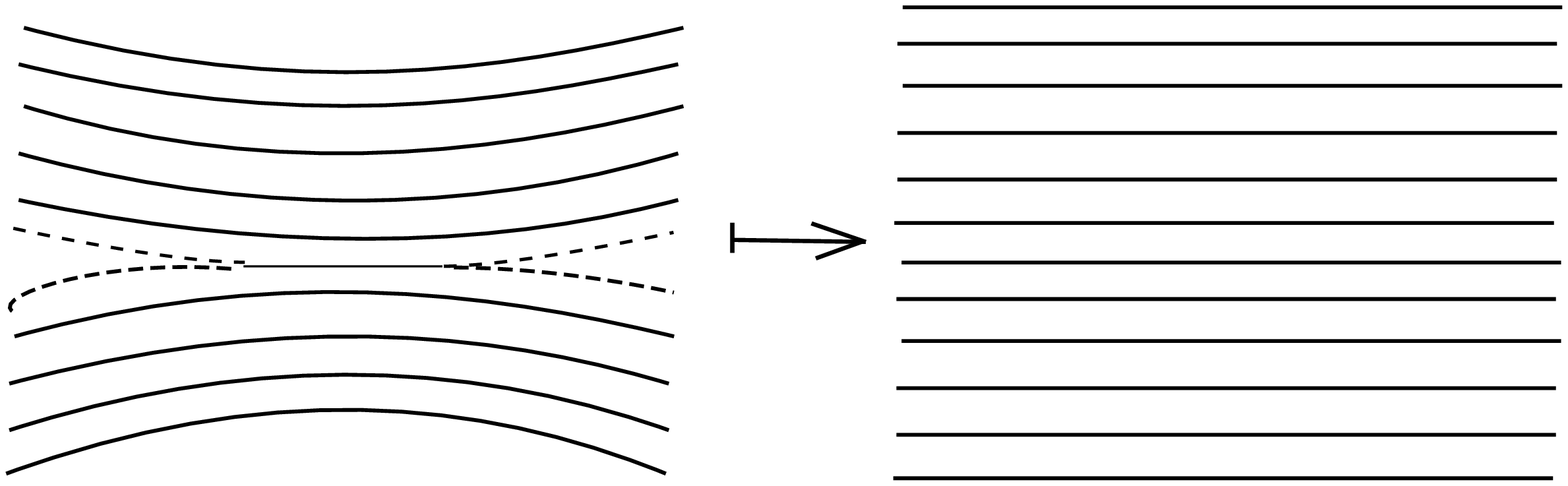} & \qquad\qquad &
\includegraphics[height = 1.9 cm]{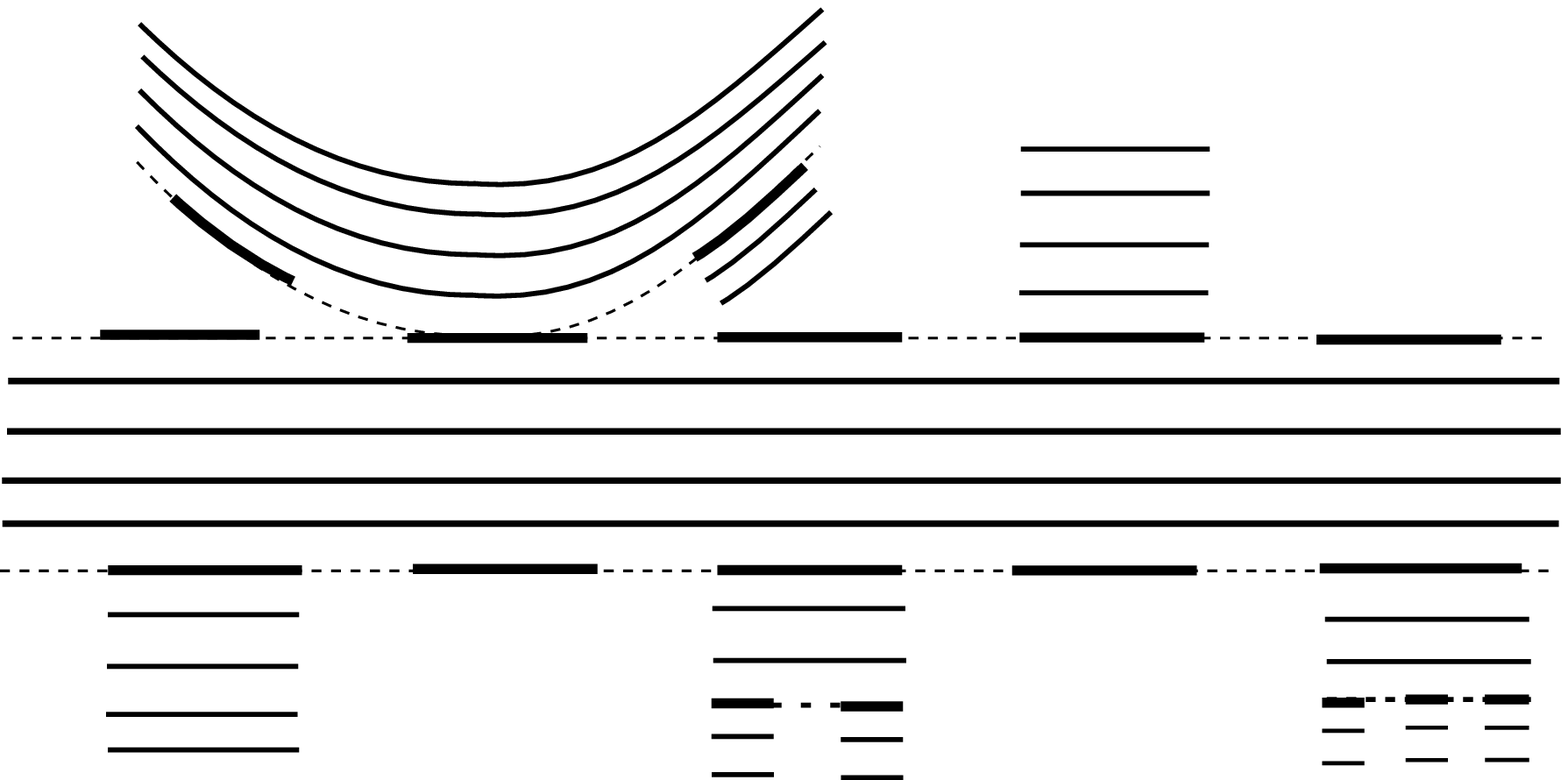} \\
(a) Non-special leaf & & (b) Special leaves
\end{tabular}
\caption{}
\label{fig:nonspec_leaf}
\end{figure}

\medskip

{\bf Reduced striped surfaces.}
A striped surface $\Sigma$ will be called \emph{reduced} whenever a leaf $\omega$ is special if and only if
one of the following conditions holds:
\begin{enumerate}
\item [1)]$\omega$ is an image of gluing of some leaves $X_{p(\gamma)}\sim Y_{q(\gamma)}$ for some $\gamma\in C$;
\item [2)] $\omega \subsetneq \partial_{-} S_{\lambda}$ or $\omega \subsetneq \partial_{+} S_{\lambda}$ for some $\lambda\in\Lambda$.
\end{enumerate}

Let $S$ be a model strip such that $\partial_{-}S = (0,1) \times {-1}$ and $\partial_{+}S = (0,1) \times {1}$.
Let also $\phi:\partial_{-}S \to \partial_{+}S$ be a homeomorphism defined by $\phi(t,-1) = (t,1)$, $t\in(0,1)$, and
$\mathcal{C} = S / \phi$ be the quotient space obtained by identifying each $x\in\partial_{-}S$ with $\phi(x) \in \partial_{+}S$.

Then $\mathcal{C}$ is a striped surface homeomorphic with a cylinder, and its canonical foliation has no special leaves.

It follows from~\cite[Theorem~3.7]{MaksymenkoPolulyakh:PGC:2015} that every striped surface (in the sense of~\eqref{equ:striped_surface}, see Remark~\ref{rem:wider_class}) is foliated homeomorphic either to $\mathcal{C}$ or to a reduced surface.

\medskip

{\bf Graph of a striped surface.}
For a reduced striped surface $\Sigma$ not foliated homeo\-morphic with $\mathcal{C}$ define an oriented graph $\Gamma(\Sigma)$ whose vertexes are strips and whose edges are special leaves.
More precisely: if $\omega = X_{p(\gamma)} \sim Y_{q(\gamma)}$ is a special leaf of $F$, $X_{p(\gamma)} \subset
\partial_{-} S_{\lambda_0}$, and $Y_{q(\gamma)} \subset \partial_{+} S_{\lambda_1}$, then we assume that $\omega$
is an \emph{edge} between \emph{vertices} $S_{\lambda_0}$ and $S_{\lambda_1}$ oriented from $S_{\lambda_1}$ to $S_{\lambda_0}$.

If $\lambda_0 = \lambda_1$, then $\omega$ correspond to a loop in $\Gamma(\Sigma)$ at $S_{\lambda_0}=S_{\lambda_1}$ being a vertex of $\Gamma(\Sigma)$.

Since a model strip may have infinitely many boundary components, we see that a graph $\Gamma(\Sigma)$ can be not locally finite.
On the other hand, it can have a finite diameter $\diam\Gamma(\Sigma)$, being the minimal non-negative integer $d$ such that every two vertices $v_1$ and $v_2$ are connected in $\Gamma(\Sigma)$ by a path consisting at most $d$ edges.

\medskip

{\bf Admissible striped surfaces.}
Recall that a family $\mathcal{V} = \{V_i\}_{i\in\Lambda}$ of subsets in a topological space $X$ is called \emph{locally finite} whenever for each $x\in X$ there exists an open neighborhood intersecting only finitely many elements from $\mathcal{V}$.

It is well known and is easy to see that \emph{a union of a locally finite family of closed subsets is closed},
e.g.~\cite[Chapter 1, \S~5.VIII]{Kuratowski:Top1:1966}.
\begin{definition}
{\rm{A model strip $S$ will be called \emph{admissible} if the closures of intervals in $\partial_{-} S$
and $\partial_{+} S$ are mutually disjoint and constitute a locally finite family in
$\rr^{2}$.}}
\end{definition}

\begin{example}\rm
A model strip with
\[
 \partial_{+} S = \bigcup\limits_{n \in \zz\setminus\{-1,0\}} \bigl(\tfrac{1}{n+1},\tfrac{1}{n}\bigr) \times {1}
\]
is not admissible, since condition 2) of Definition~\ref{df:model_strip} fails.
\end{example}

It will be convenient to use the following notation:
\iffalse
\begin{align*}
[0] &= \varnothing, &
[n] &= \{1,2,\ldots, n\}, &
-\nn &= \{-1,-2,\ldots\}.
\end{align*}
\fi
$$
[0] = \varnothing, \quad
[n] = \{1,2,\ldots, n\}, \quad
-\nn = \{-1,-2,\ldots\}.
$$

Let also $J_{\iind} = (\iind, \iind+0.5)$, $\iind\in\zz$, and for a subset $\Iind \subset \zz$
denote
\[
\typeGen{\Iind} =  \bigcup\limits_{\iind\in\Iind} J_{\iind}.
\]

In particular, consider the following collections of mutually disjoint open
intervals:
\begin{align*}
\typeFinite{n}=  & \bigcup\limits_{\iind=1}^{n} (\iind, \iind+0.5), \quad n=0,1,\ldots, &
\typeNatural=    &  \bigcup\limits_{\iind \in \nn} (\iind, \iind+0.5),  \\ %\bigcup\limits_{\iind=1}^{+\infty} (\iind, \iind+0.5),  \\
\typeNaturalInv= & \bigcup\limits_{-\iind\in\nn} (\iind, \iind+0.5), & %\bigcup\limits_{\iind=-\infty}^{-1} (\iind, \iind+0.5), &
\typeIntegers=   &  \bigcup\limits_{\iind\in\zz} (\iind,
\iind+0.5),
%\bigcup\limits_{\iind=-\infty}^{+\infty} (\iind, \iind+0.5).
\end{align*}
which will be called \emph{standard}.
The following easy lemma is left for the reader.

\begin{lemma}
Let $S$ be an admissible model strip.
Then there exists a homeomorphism $h:\rr^2\to\rr^2$ preserving each line $\rr\times t$, $t\in(-1,1)$, with its orientation, and such that $h(S)$ is a model strip with $\partial_{-}h(S) = \typeGen{\alpha} \times\{-1\}$ and $\partial_{+}h(S)=\typeGen{\beta} \times\{1\}$, where $\typeGen{\alpha}$ and $\typeGen{\beta}$ are standard collections of intervals, i.e. $\alpha, \beta\in\{[0], [1], \ldots, \ntN, \ntNInv, \ntZ\}$, see Figure~\ref{fig:a_types}.
Moreover, $\alpha$ and $\beta$ do not depend on a particular choice of such $h$.
\end{lemma}
\begin{figure}[h]
\begin{tabular}{ccc}
\includegraphics[height = 1.2cm]{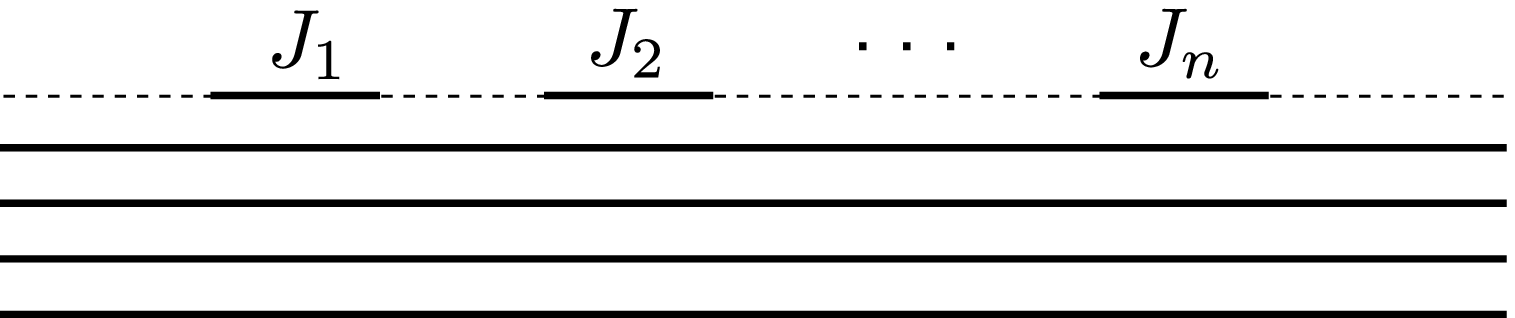} & \qquad &
\includegraphics[height = 1.2cm]{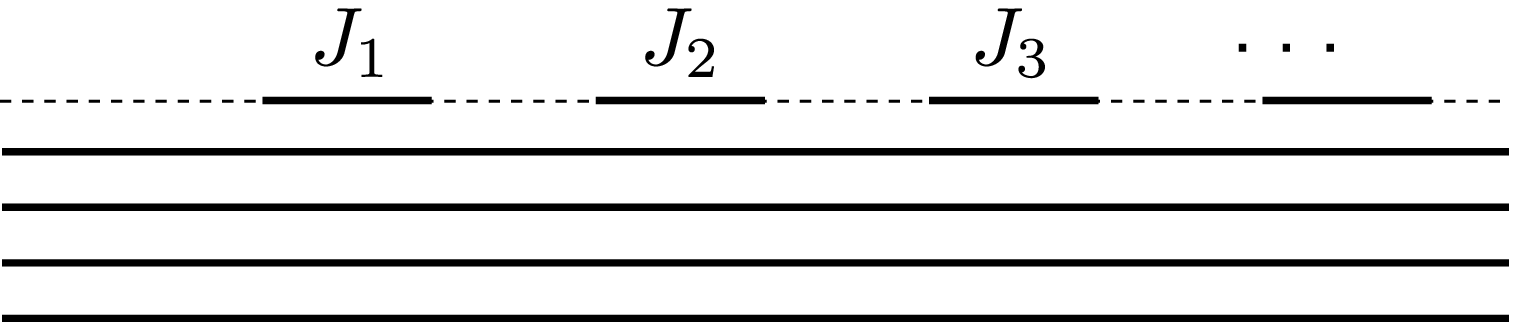} \\
$\typeFinite{n}$ & & $\typeNatural$ \\ \\
\includegraphics[height = 1.2cm]{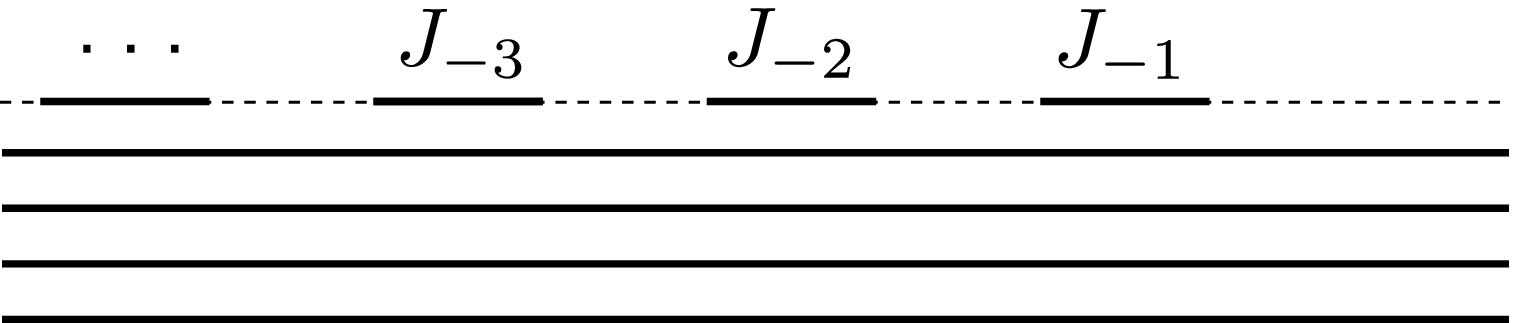} & &
\includegraphics[height = 1.2cm]{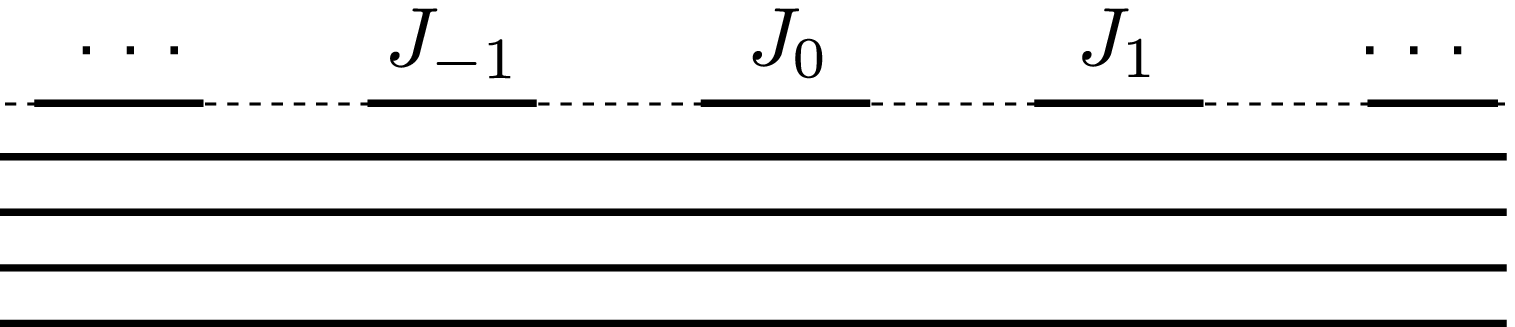} \\
$\typeNaturalInv$ & & $\typeIntegers$
\end{tabular}
\caption{Types of $\partial_{+}S$}
\label{fig:a_types}
\end{figure}

Thus for an admissible model strip $S$ its foliated topological type is determined by the ordinal type of collections of boundary intervals in $\partial_{-}S$ and $\partial_{+}S$.

\section{Wreath products}
Let $H$ and $S$ be two groups.
Denote by $Map(H,S)$ the group of all \emph{maps} (not necessarily homomorphisms) $\varphi: H \rightarrow S$ with respect to the point-wise multiplication.
Then the group $H$ acts on $Map(H,S)$ by the following rule: the result of the action of $\varphi\in Map(H,S)$ on $h\in H$ is the composition map:
\[
\varphi \circ h :  H \longrightarrow H \longrightarrow S.
\]
The semidirect product $Map(H,S) \rtimes H$ corresponding to this action will be denoted by $S \wr H$ and called the \emph{wreath product} of $S$ and $H$.
Thus
\[
S \wr H= Map(H,S) \rtimes H
\]
is the Cartesian product $Map(H,S) \times H$ with the multiplication given by the formula
\[
 (\varphi_1, h_1) \cdot (\varphi_2, h_2)
= \bigl((\varphi_1 \circ h_2)\cdot \varphi_2, h_1 \cdot h_2 \bigr)
\]
for $(\varphi_1, h_1), (\varphi_2, h_2) \in Map(H,S) \rtimes H$.

Let  $\varepsilon: H \rightarrow S$ be the constant map into the unit of $S$.
Then the pair $(\varepsilon, \id_H)$ is the unit element of $S \wr H$.
Moreover, if $(\varphi,h) \in S \wr H$ and $\varphi^{-1} \in Map(H,S)$ is the point-wise inverse of $\varphi$, then $(\varphi^{-1} \circ h^{-1}, h^{-1})$ is the inverse of $(\varphi,h)$ in $S \wr H$.

We also have the following exact sequence:
\[
1 \to Map(H,S) \xrightarrow{~~i~~} S \wr H \xrightarrow{~~\pi~~} H \to 1,
\]
where $i(\varphi)=(\varphi,e)$, $e$ is the unit of $H$,  and $\pi(\varphi,h)=h$.
Moreover, $\pi$ admits a section $s: H \rightarrow S \wr H$ defined by $s(h)=(\varepsilon, h)$.

\section{Main result}
{\bf Homeotopy group of a canonical foliation.}
Let $\Sigma$ be striped surface with a canonical foliation $F$.
Denote by $\HF$ the groups of all foliated homeomorphisms $h:\Sigma\to\Sigma$, i.e.\! homeomorphisms mapping leaves of $F$ onto leaves.
We will endow $\HF$ with the corresponding compact open topology.

Recall that all leaves of $F$ are oriented.
Then we denote by $\Hplus$ the subgroup of $\HF$ consisting of homeomorphisms $h: \Sigma \to \Sigma$ such that for each leaf $\omega$ the restriction map $h:\omega\to h(\omega)$ is orientation preserving.

Let $H_0^{+}(F)$ be the identity path component of $\Hplus$.
It consists of all $h\in\Hplus$ isotopic to $\id_{\Sigma}$ in $\Hplus$.
Then $H_0^{+}(F)$ is a normal subgroup of $\Hplus$, and the corres\-ponding quotient
$$\pi _0 H^{+}(F)=H^{+}(F) / H_0^{+}(F)$$
will be called the \emph{homeotopy} group of $F$.

\medskip

{\bf Class $\mathfrak{F}$.}
Denote by $\mathfrak{F}$ the class of striped surfaces
\[
\Sigma = \bigsqcup \limits_{\lambda \in \Lambda} S_{\lambda} \bigr/ \{X_{p(\gamma)} \stackrel{\varphi_{\gamma}}{\sim} Y_{q(\gamma)} \}
\]
satisfying the following conditions:
\begin{enumerate}
\item[1)] each $S_{\lambda}$, $\lambda \in \Lambda$, is
admissible,
\iffalse
\begin{align*}
 \partial_{-} S_{\lambda} &= J_1\times\{-1\} = \typeFinite{1} \times\{-1\}, &
  \partial_{+} S_{\lambda} &= \typeGen{\Iind_{\lambda}}\times\{1\},
\end{align*}
\fi
$$
 \partial_{-} S_{\lambda} = J_1\times\{-1\} = \typeFinite{1} \times\{-1\},
 \quad
  \partial_{+} S_{\lambda} = \typeGen{\Iind_{\lambda}}\times\{1\},
$$
where $\Iind_{\lambda}$ coincides with one of the standard collections $\typeFinite{n}$, $\typeNatural$, $\typeNaturalInv$, or $\typeIntegers$;
\item[2)] the graph $\Gamma(\Sigma)$ is connected and has a finite diameter and no cycles.
\end{enumerate}

In particular, if $\Sigma \in \mathfrak{F}$, then each model strip $S_{\lambda}$ of $\Sigma$ regarded as a vertex of $\Gamma(\Sigma)$ has at most one incoming edge and at most countably many outcoming edges linearly ordered with respect to $\Iind_{\lambda}$.

Since $\Gamma(\Sigma)$ is connected and has a finite diameter and no cycles, it follows that there exists a unique vertex having no incoming edges.
We will call this vertex a \emph{root} and the corresponding strip a \emph{root} strip.

Thus every surface $\Sigma \in \mathfrak{F}$ of diameter $d$ can be represented as follows, see Figure~\ref{ris:image3}:
\begin{equation}\label{equ:striped_surface_F}
\Sigma \ = \ S \ \mathop{\cup}\limits_{\partial_{+}S} \ \Bigl(  \bigcup \limits_{\iind \in \Iind} \Sigma_{\iind} \Bigr),
\end{equation}
where
\begin{itemize} %[leftmargin=*]
\item $S$ is a root strip of $\Sigma$,
\iffalse
\begin{align*}
\partial_{-}S &= J_1\times\{-1\}, &
\partial_{+}S &= \mathop{\cup}_{\iind\in\Iind} J_{\iind} \times \{1\},
\end{align*}
\fi
$$
\partial_{-}S = J_1\times\{-1\}, \quad
\partial_{+}S = \mathop{\cup}_{\iind\in\Iind} J_{\iind} \times \{1\},
$$
where $\Iind \in \{[0], [1],\ldots, \ntN, \ntNInv, \ntZ\}$.

\item
$\Sigma_{\iind}$ is either empty or it is a striped surface belonging to $\mathfrak{F}$ and its graph $\Gamma(\Sigma_{\iind})$ has diameter less than $d$.

\item
Suppose $\Sigma_{\iind}$ is non-empty and let $S_{\iind}$ be the root strip of $\Sigma_{\iind}$.
Then $\partial_{-}S_{\iind}=J_1\times\{-1\}$ is glued to the boundary interval $J_{\iind} \times \{1\}$ of $\partial_{+}S$ by the homeomorphism
\[
\varphi:J_1\equiv(1,1.5) \longrightarrow  J_{\iind} \equiv (\iind, \iind+0.5),
\quad
\varphi(t) = t+\iind-1.
\]
\end{itemize}

\begin{figure}[h]
\includegraphics[height = 2.5 cm]{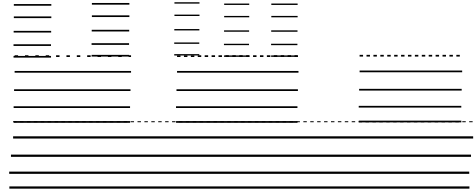}
\caption{A striped surface $\Sigma \in \mathfrak{F}$ whose graph $\Gamma(\Sigma)$ has diameter $3$}
\label{ris:image3}
\end{figure}

Obviously, $\Sigma \in \mathfrak{F}$ is a connected and simply connected non-compact surface.
Therefore it follows from \cite{Epstein:AM:1966} that the interior of $\Sigma$ is homeomorphic to $\rr^{2}$.

The class of homeotopy groups of foliations on striped surfaces which belongs to the class $\mathfrak{F}$ will
be denoted by
%$\classHomeotopy$
$\mathcal P$, i.e.
\iffalse
$$ \classHomeotopy = \lbrace\pi _0 H^{+}(F) \  | \ F \text{ is a canonical foliation of some  striped surface }
\Sigma \in \mathfrak{F} \rbrace.$$
\fi
$$ \mathcal P = \lbrace\pi _0 H^{+}(F) \  | \ F \text{ is a canonical foliation of some  striped surface }
\Sigma \in \mathfrak{F} \rbrace.$$

We will also define  another class of groups
%$\classGroups$
$\mathcal G$.

\begin{definition}\label{def:classZ}
{\rm{Let $\mathcal G$ be the minimal class of groups satisfying the following conditions:
\begin{enumerate}
\item[\rm 1)] $\lbrace 1 \rbrace \in \mathcal G$;
\item[\rm 2)]  if $A_i \in \mathcal G$ for $i\in\nn$, then $\prod\limits_{i \in \nn} A_i \in \mathcal G$;
\item[\rm 3)] if $ A \in \mathcal G$, then  $A \wr \zz\in \mathcal G$.
\end{enumerate} }}
\end{definition}

\begin{lemma}\label{lm:charact_classGroups}
A group $G$ belongs to $\mathcal G$ if and only if it can be obtained from the unit group $\lbrace 1\rbrace$ by a composition of finitely many operations of the following types:
\begin{itemize}
\item[\rm(a)]
countable direct products;
\item[\rm(b)]
wreath product with the group $\zz$.
\end{itemize}
\end{lemma}
\begin{proof}
Let $\mathcal G_0$ be the class of groups $G$ which can be obtained from the unit group $\lbrace 1\rbrace$ by a
composition of finitely many operations of types (a) and (b).
Then any class of groups satisfying conditions 1)--3) of Definition~\ref{def:classZ} contains $\mathcal G_0$,
whence $\mathcal G_0 \subset \mathcal G$.
On the other hand, $\mathcal G_0$ also satisfies conditions 1)--3) of Definition~\ref{def:classZ},
whence $\mathcal G \subset \mathcal G_0$ as well.
\end{proof}

Every representation $\xi(G)$ of $G$ as a composition of operations (a) and (b) will be called
\emph{a representation of $G$ in the class $\mathcal G$}.
Such a representation is not unique.
For example,
\begin{equation}\label{equ:repr_for_Z}
 \zz \ \cong \ \lbrace 1 \rbrace \wr \zz  \ \cong \ 1 \times (1 \wr \zz) \ \cong \ (1\times 1 \times 1) \wr \zz.
\end{equation}

\begin{definition}\label{df:height}
{\rm{The \emph{height} of a representation $\xi(G)$ of $G$ in the class $\mathcal G$ is a non-negative integer defined inductively as follows:
\begin{enumerate}
\item[\rm1)] $h(\lbrace 1 \rbrace)= 0$;
\item[\rm2)] $h(\xi(G) \wr \zz)= 1 + h(\xi(G))$;
\item[\rm3)] $h\left(  \prod\limits_{i \in  \Lambda} \xi(A_i) \right) = 1 + \max\limits_{i} \lbrace h (\xi(A_i)) \rbrace $.
\end{enumerate}}}
\end{definition}
\begin{example}
{\rm{Below are examples of representations of groups $\{1\}$, $\zz$ and $\zz\wr\zz$ in the class $\mathcal G$ and their heights:
\begin{align*}
&h(\{1\}) = 0, &
&h(\{1\}\times\{1\}) = 1, \\
&h(\{1\} \wr \zz) = 1, &
&h((\{1\}\times\{1\}) \wr \zz) = 2, \\
&h\bigl(\, (\{1\}\wr\zz)\,\times\,(\{1\}\wr\zz)\, \bigr) = 2, &
&h\bigl(\,((\{1\}\times\{1\})\wr\zz) \, \times \, (\{1\}\wr\zz)\,\bigr) = 3.
\end{align*}}}
\end{example}

\medskip

Let $\mathcal G' \subset \mathcal G$ be a subclass of $\mathcal G$ consisting of groups admitting a
representation of finite height in $\mathcal G$.
The aim of the present paper is to prove the following theorem:
\begin{theorem}\label{th:main_theorem}
Classes $\mathcal P$ and $\mathcal G'$ coincide.
\end{theorem}

In other words, a group $G$ is isomorphic with a homeotopy group $H^{+}(F)$ of some striped surface $\Sigma\in\mathfrak{F}$ with a canonical foliation $F$ if and only if $G$ can be obtained from the unit group $\lbrace 1\rbrace$ by a composition of finitely many operations of types (a) and (b) of Lemma~\ref{lm:charact_classGroups}.

\section{Preliminaries}
Let $\Sigma$ be a striped surface belonging to $\mathfrak{F}$ presented in the form~\eqref{equ:striped_surface_F}, and $S$ be the root strip of $\Sigma$.
We will use coordinates $(x,y)$ from the chart for $S$, so we can assume that $\partial_{+}S = \cup_{\iind\in\Iind} J_{\iind} \times \{1\}$.

Notice that if $h\in\Hplus$, then $h(S)=S$, whence there exists a unique number $\eta(h) \in \zz$ such that in the chart for $S$ we have that
\[
h(J_{\iind}\times \{1\}) = J_{\iind + \eta(h)}\times \{1\}
\]
for all $\iind \in \Iind$.
One can easily check that the correspondence $h\mapsto\eta(h)$ is a homomorphism
\begin{equation}\label{equ:eta}
\eta: \Hplus \rightarrow \zz.
\end{equation}
Obviously, $\eta$ can be a non-zero homomorphism only when $\Iind =\zz$.

Consider the following two subgroups of $\Hplus$:
\begin{align*}
Q_S &= \left\lbrace  h \in \Hplus \ | \ h(\omega)=\omega, \text{for each leaf $\omega$ of $F \subset S$ }  \right\rbrace, \\
\HifixS[S] &= \left\lbrace  h \in \Hplus \ | \ h|_{S} = \id|_{S}  \right\rbrace.
\end{align*}
It is evident that
\begin{equation}\label{equ:subgrouprs_of_ker_eta}
\HifixS[S] \ \subset \ Q_S \ \subset \ker(\eta).
\end{equation}
\begin{lemma}\label{lm:embed_ker}
Embeddings~\eqref{equ:subgrouprs_of_ker_eta} are homotopy equivalences.
\end{lemma}

\begin{proof}
First we will construct a deformation of $\ker(\eta)$ into $Q_S$.
Let $h\in\ker(\eta)$.
Since $h(S)=S$, it follows that $h$ interchanges leaves of $F$.
In the coordinates $(x,y)$ in the chart for $S$ these leaves are the lines $y=\mathrm{const}$, whence
\[
h(x, y) = \left( \alpha(x, y), \beta(y) \right),
\]
where $\alpha:S\to\rr$ and $\beta:[-1,1]\to[-1,1]$ are continuous functions such that for each $y\in(0,1)$ the correspondence   $x\mapsto \alpha(x, y)$ is a preserving orientation homeomorphism $\rr\to\rr$.

Then $h\in Q_S$ iff $\beta(y)=y$ for all $y\in[0,1]$. Define the
map $H:\ker(\eta)\times[0,1]\to\ker(\eta)$ by the formula
\[
H(h,t)(z) =
\begin{cases}
\left(  \alpha(x,y), (1-t)\beta(y)+ty \right), & z=(x,y) \in S, \\
z, & z \in  \Sigma \setminus S.
\end{cases}
\]
One can easily check that $H_0=\id_{\ker(\eta)}$, $H_t(Q_S) \subset Q_S$ for all $t\in[0,1]$, and $H(h,1) \in Q_S$.
Hence $H$ is a deformation of $\ker(\eta)$ into $Q_S$, and so the inclusion $Q_S\subset\ker(\eta)$ is a homotopy equivalence.

Similarly, let $h\in Q_S$, so
\[
h(x, y) = \left( \alpha(x, y), y \right)
\]
for all $(x,y)\in S$.
Notice that $h\in \HifixS[S]$ iff $\alpha(x,y)=x$ and $\beta(y)=y$ for all $(x,y)\in S$.

Let
\[
h(x,y) = \left( \alpha_{\iind}(x, y), \beta_{\iind}(y) \right)
\]
be the restriction of $h$ onto root strip $S_{\iind}$ of $\Sigma_{\iind}$ in the corresponding chart of $S_{\iind}$.
Since $\partial_{-}S_{\iind} = J_1\times\{-1\}$, we see that if $h\in\HifixS[S]$, then $\alpha_{\iind}(x,-1)=x$ for all $x\in J_1$ and $\iind\in\Iind$.

Fix a continuous function $\varepsilon:[-1,1]\to[0,1]$ such that
\[
\varepsilon(y) =
\begin{cases}
0, & y \in (-1,-0.8), \\
1, & y \in (0,1)
\end{cases}
\]
and define the following homotopy $G:Q_S \times [0,1] \to Q_S$ by
\[
G(h,t)(z) =
\begin{cases}
\left( (1-t) \alpha(x, y)+tx, y \right), & z=(x,y) \in S, \\
\bigl( (1-t \varepsilon(y) )\alpha_{\iind}(x, y) +t\varepsilon(y)x, \beta(y) \bigr), & z=(x,y) \in S_{\iind}, \\
z& z\not\in S \, \cup \,\bigl(\cup_{\iind\in\Iind} S_{\iind} \bigr).
\end{cases}
\]
Since $\partial_{-}S_{\iind}$ is glued to the boundary component $J_i\times \{1\}$ by an affine homeomorphism, and the formulas for $G$ are affine for each fixed $t$ and $y$, it follows that those formulas agree on $J_i\times \{1\}$ and $\partial_{-}S_{\iind}$, c.f.~\cite{MaksymenkoPolulyakh:PGC:2015}.
This implies that $G$ is a continuous map.

Moreover, one can easily check that $G_0=\id_{Q_S}$, $G_t(\HifixS[S])\subset \HifixS[S]$ for all $t\in[0,1]$, and $G_1(Q_S) \subset \HifixS[S]$.
Hence $G$ is a deformation of $Q_S$ into $\HifixS[S]$, and therefore the inclusion $\HifixS[S] \subset Q_S$ is a homotopy equivalence as well.
\end{proof}

Suppose $\Sigma_{\iind}$ is non-empty for some $\iind\in\Iind$.
Let $F_{\iind}$ be the canonical foliation on $\Sigma_{\iind}$ and $S_{\iind}$ be the root strip of $\Sigma_{\iind}$.
We will denote by $\HFdSi{\iind}$ the subgroup of $H^{+}(F_{\iind})$ consisting of homeomorphisms fixed on $\partial_{-}S_{\iind}$.

If $\Sigma_{\iind} = \varnothing$, then we will assume that $\HFdSi{\iind} = \{1\}$.

\begin{lemma}\label{lm:hfix_ker}
We have an isomorphism
\[ \pi_0 \ker(\eta) \cong \prod \limits_{\iind \in \Iind}   \pi_{0}\HFdSi{\iind}.\]
\end{lemma}
\begin{proof}
Evidently, we have a canonical isomorphism
\iffalse
\begin{align*}
\alpha:& \HifixS[S] \cong \prod \limits_{\iind \in \Iind} \HFdSi{\iind}, &
\alpha(h) = ( h|_{\Sigma_{\iind}} )_{\iind\in\Iind}.
\end{align*}
\fi
$$
\alpha: \HifixS[S] \cong \prod \limits_{\iind \in \Iind} \HFdSi{\iind},
\quad
\alpha(h) = ( h|_{\Sigma_{\iind}} )_{\iind\in\Iind}.
$$
Then from Lemma~\ref{lm:embed_ker} we get the following sequence of isomorphisms:
\[ \pi_0 \ker(\eta) \cong \pi_0 \HifixS[S] \cong
\pi_0  \prod \limits_{\iind \in \Iind} \HFdSi{\iind}   = \prod \limits_{\iind \in \Iind} \pi_{0}\HFdSi{\iind}. \]
Lemma is proved.
\end{proof}

\begin{theorem} \label{structure_theorem}
{\rm1)}
If $\eta$ is zero homomorphism, then the group $\pi_0 \Hplus$ is isomorphic to $\prod \limits_{\iind \in \Iind} \pi_0 \HFdSi{\iind}$.

{\rm2)}
Suppose the image of $\eta$ is $k \zz$ for some $k\geq1$, so $\Iind =\zz$.
Then the group $\pi_0\Hplus$ is isomorphic to $\left( \prod \limits_{\iind= 0}^{k-1} \pi_0  \HFdSi{\iind} \right) \wr \zz$.
\end{theorem}
\begin{proof}
1)~The assumption that $\eta$ is zero homomorphism means that $\Hplus = \ker(\eta)$, whence we get from Lemma~\ref{lm:hfix_ker} that
\[
\pi_0 \Hplus \cong \prod \limits_{\iind \in \Iind} \pi_0  \HFdSi{\iind}.
\]

2)~Suppose $\Im \eta = k \zz$.
Then we have an \emph{epimorphism} $\widehat{\eta}: \Hplus \rightarrow \zz$
defined by $\widehat{\eta}(h) = \eta(h) / k$ and such that
\[
h(\Sigma_r) = \Sigma_{r+k\cdot \widehat{\eta}(h)}, \quad r = 0,1,\ldots,k-1.
\]

Let
\iffalse
\begin{align*}
\Xsp &= \bigcup_{\iind=0}^{k-1} \Sigma_{\iind}, &
\partial_{-}\Xsp &= \bigcup_{\iind=0}^{k-1} \partial_{-}S_{\iind},
\end{align*}
\fi
$$
\Xsp = \bigcup_{\iind=0}^{k-1} \Sigma_{\iind}, \quad
\partial_{-}\Xsp = \bigcup_{\iind=0}^{k-1} \partial_{-}S_{\iind},
$$
and $F_{\Xsp}$ be the oriented foliation on $\Xsp$ induced by $F$.
Denote by $\HFX$ the group of homeomorphisms of $\Xsp$ fixed on $\partial_{-}\Xsp$ and mapping leaves of $F_{\Xsp}$ onto leaves and preserving their orientation.
Then we have a natural isomorphism
\[
\prod \limits_{\iind=0}^{k-1} \HFdSi{\iind}  \ \cong \ \HFX
\]
which yields an isomorphism
\[
\prod \limits_{\iind=0}^{k-1} \pi_0\HFdSi{\iind}  \ \cong \ \pi_0\HFX.
\]

Therefore for the proof of Theorem~\ref{structure_theorem} we should construct an
isomorphism
\[
\beta: \pi_0\Hplus \ \longrightarrow \ \pi_0\HFX \ \wr \ \zz \ \equiv \  Map \Bigl( \zz,\pi_0\HFX \Bigr) \rtimes \zz.
\]

Fix any $g \in \Hplus$ with $\widehat{\eta}(g)=1$.
Then \[ g^{-\widehat{\eta}(h)} \circ h \left( \Sigma_{\iind} \right) = \Sigma_{\iind},\]
for all $h\in\Hplus$ and $\iind \in \zz$, whence $g^{-\widehat{\eta}(h)} \circ h \in \ker(\eta)$.
Thus we get a well-defined function
\[
\varphi_{h}: \zz \rightarrow \pi_0\HFX, %\prod \limits_{i= 0}^{k-1} \pi_0 \HfixdS[i],
\quad
\varphi_{h}(j)= \left[ g^{-j-\widehat{\eta}(h)} \circ h \circ g^{j} \bigl|_{\Xsp} \right].
\]
Define the following map:
\[
\beta: \pi_0\Hplus \longrightarrow \pi_0\HFX
\]
by the formula
\[
\beta(h) = \left( \varphi_{h}, \widehat{\eta}(h) \right), \quad h \in \pi_0  \Hplus.
\]
We claim that $\beta$ is an isomorphism.
First notice that the composition operation in $\HFX\wr\zz$ is given by the following rule:
\[
(\varphi_{h_1}, n)\cdot (\varphi_{h_2}, m)= (\varphi_{h_1}^{m} \cdot \varphi_{h_2}, n+m),
\]
where $\varphi_{h}^{m}(j)= \varphi_{h} (j+m)$.

{\bf Proof that $\beta $ is a homomorphism.}
Let $h_1, h_2 \in \Hplus$. Then
\begin{align*}
\beta(h_1) \circ \beta (h_2)
&= \bigl( \varphi_{h_1}, \ \widehat{\eta} (h_1) \bigr) \cdot \bigl(\varphi_{h_2}, \ \widehat{\eta} (h_2) \bigr) \\
&=\bigl(\varphi_{h_1}^{\widehat{\eta}(h_2)} \cdot \varphi_{h_2}, \ \widehat{\eta} (h_1) + \widehat{\eta} (h_2) \bigr) \\
&=\Bigl(\!\bigl[ g^{- j - \widehat{\eta} (h_1) - \widehat{\eta} (h_2)} \circ h_1 \circ g^{j+ \widehat{\eta} (h_2)} \circ g^{- j - \widehat{\eta} (h_2)} \circ h_2 \circ g^{j} |_{\Xsp}\bigr], \  \widehat{\eta} (h_1 \circ h_2)\!\Bigr) \\
&=\left(\!\bigl[ g^{- j - \widehat{\eta} (h_1 \circ h_2)} \circ h_1 \circ h_2 \circ g^{j} |_{\Xsp} \bigr], \ \widehat{\eta} (h_1 \circ h_2)\!\right) \\
&=\bigl( \varphi_{h_1 \circ h_2}, \ \widehat{\eta} (h_1 \circ h_2) \bigr) = \beta(h_1 \circ h_2).
\end{align*}

{\bf Proof that $\beta$ is injective.}
Let $h\in\Hplus$ be such that $[h] \in \ker \beta$.
We should prove that $h$ is isotopic in $\Hplus$ to $\id_{\Sigma}$.

The assumption $[h] \in \ker \beta$ means that $\beta (h) = (\varphi_h, \widehat{\eta}(h)) = (\varepsilon, 0)$, where $\varepsilon: \zz \rightarrow [\id_{\Xsp}]$ is the constant map into the unit of  $\pi_0\HFX$.
In particular, since $\eta(h)=0$, we get from Lemma~\ref{lm:embed_ker} that $h$ is isotopic in $\Hplus$ to a homeomorphism fixed on $S$.
Therefore we can assume that $h$ itself is fixed on $S$, that is $h\in \HifixS[S]$.
Then
\begin{equation}\label{equ:func_phih}
\varphi_{h}(j)= \left[g^{-j} \circ h \circ g^{j}|_{\Xsp} \right] = \varepsilon (j)=\left[ \id_{\Xsp} \right] \in \pi_0\HFX
\end{equation}
for each $j\in\zz$.
In other words, $g^{-j} \circ h \circ g^{j}|_{\Xsp}$ is isotopic to $\id_{\Xsp}$ relatively $\partial_{-}\Xsp$.

It suffices to prove that for each $\iind \in \zz$ the restriction $h|_{\Sigma_{\iind}}$ is isotopic in $H^{+}(F_{\iind}, \partial_{-}S_{\iind})$ to $\id_{\Sigma_i}$ relatively to $\partial_{-} S_{\iind}$.

Write $\iind=r+jk$ for a unique $r\in\{ 0,k-1\}$.
Then we have the following commutative diagram:

\iffalse

\[
\begin{CD}
\Sigma_{\iind}=\Sigma_{r+jk} @>{g^{-j}}>> \Sigma_r \\
@V{h}VV @VV{g^{-j} \circ h \circ g^{j}}V \\
\Sigma_{\iind}=\Sigma_{r+jk} @>{g^{-j}}>> \Sigma_r
\end{CD}
\]

\fi

Therefore, we get from~\eqref{equ:func_phih} that $\left[ h |_{\Sigma_i} \right]=\left[ \id_{\Sigma_i} \right] \in H^{+}(F_{\iind}, \partial_{-}S_{\iind})$.
Hence $h$ is isotopic ot $\id_{\Sigma}$ in $\Hplus$.

{\bf Proof that $\beta$ is surjective.}
Let $(\varphi,n)\in\pi_0\HFX\wr\zz$.
For each $j\in\zz$ fix a homeomorphism $h_{j} \in \HFX$ such that $[h_j] = \phi(j) \in \pi_0\HFX$.
Now define the following homeomorphism $\widehat{h}$ of $\Sigma$ by the formula:
\[
\widehat{h} =
\begin{cases}
\id_{S}, & \text{on} \ S, \\
[g^{j}\circ h_j\circ g^{-j}] & \text{on} \ g^{j}(\Xsp)
\end{cases}
\]
and put $h = g^{n}\circ \widehat{h}$.
Then it is easy to see that $\beta([h]) = (\phi,n)$, whence $\beta$ is surjective.
Thus $\beta$ is an isomorphism.
\end{proof}

\section{Proof of Theorem~\ref{th:main_theorem}}
We should prove that $\mathcal P=\mathcal G'$.

\textbf{1.} First we will show that $\mathcal G'\subset
\mathcal P$.

Let $G \in \mathcal G'$, so $G$ has a representation $\xi(G)$ in
the class $\mathcal G$ of finite height $k=h(\xi(G))$. We have to
show that there exists a striped surface $\Sigma\in\mathfrak{F}$
with canonical foliation $F$ such that $G\cong\pi_0\Hplus$.

If $k = h(\xi(G)) = 0$, then $G$ is the unit group $\{1\}$ and $\xi (G) = \lbrace 1 \rbrace$.
Let $S$ be an admissible model strip with $\partial_{-}S=\typeFinite{1}\times\{-1\}$ and $\partial_{+}S = \varnothing$.
Then $S\in\mathfrak{F}$.
Let also $F$ be the canonical foliation on $S$.
Then
\[\pi _0 H^{+}(F) = \lbrace 1 \rbrace = G,\]
i.e. $G \in \mathcal P$.

Suppose that we have established our statement for all $k$ being less than some $\bar{k}>0$.
Let us prove it for $k=\bar{k}$.
It follows from Definition~\ref{df:height} that either
\begin{enumerate}
\item [(i)]  $\xi(G)= \prod \limits_{i \in \nn} A_i$ where each group $A_i$ has a representation $\xi(A_i)$ in the class $\mathcal G$ of height $h(\xi(A_i))<k$, or
\item [(ii)]   $\xi(G)=  A \wr \zz$, and $A$ has a representation $\xi(A)$ in the class $\mathcal G$ of height $h(\xi(A))< k$.
\end{enumerate}

In the case (i) due to the inductive assumption for each $i\in\nn$ there exists a striped surface $\Sigma_i\in\mathfrak{F}$ with foliations $F_i$ such that $A_i =\pi_0 H^{+}(F_i)$. % and $\diam\Gamma(\Sigma_i) = h(\xi(A_i))$.

Let $S$ be an admissible model strip with $\partial_{-}S = \typeFinite{1}\times\{-1\}$ and $\partial_{+}S  = \typeNatural \times \{1\}$, and $S_i$ be the root strip of $\Sigma_i$, $i\in\nn$.
Define the striped surface
\[
\Sigma \ = \ S \ \mathop{\cup}\limits_{\partial_{+} S }\  \Bigl(\mathop{\cup}\limits_{i\in\nn} \Sigma_i\Bigr)
\]
obtained by identifying $\partial_{-}S_i \subset \Sigma_i$ with $J_i\times\{1\} \subset \partial_{+} S$.
Then by Theorem~\ref{structure_theorem} $\eta$ is a trivial homomorphism, and
$\pi_0H^{+}(F) \cong \prod_{i\in\nn} \pi_0H^{+}(F_i) \cong \prod \limits_{i \in \nn} A_i \cong G$.
So $G \in \mathcal P$.

In the case (ii) again by inductive assumption there exists a striped surface $\widehat{\Sigma}\in\mathfrak{F}$ with a canonical foliation $\widehat{F}$ such that $A =\pi _0 H^{+}(\widehat{F})$. % and $\diam\Gamma(\widehat{\Sigma}) = h(\xi(A))$.

Take countably many copies $\widehat{\Sigma}_i$, $i\in\zz$, of $\widehat{\Sigma}$.
Let $\widehat{S}_i$ be the root strip of $\widehat{\Sigma}_i$ and $\widehat{F}_i$ be the canonical foliation on $\widehat{\Sigma}_i$.

Let also $S$ be an admissible model strip with $\partial_{-}S = \typeFinite{1}\times\{-1\}$ and $\partial_{+}S  = \typeIntegers \times \{1\}$.
Define the following striped surface:
\[
\Sigma \ = \  S \ \mathop{\cup}\limits_{\partial_{+} S} \  \Bigl(\, \mathop{\cup}\limits_{i\in\nn}
\widehat{\Sigma}_i \,\Bigr).
\]
Obtained by gluing each $\widehat{\Sigma}_i$ to $S$ by identifying $\partial_{-}\widehat{S}_i \subset \widehat{\Sigma}_i$ with $J_i\times\{1\} \subset \partial_{+} S$, $i\in\zz$.

Then for every pair $i,j\in\zz$ there exists $h\in H^{+}(F)$ such that $h(\widehat{\Sigma}_i) = \widehat{\Sigma}_j$, whence the homomorphism $\eta$, see~\eqref{equ:eta} is surjective.
Hence by Theorem~\ref{structure_theorem}
\[ \pi_0H^{+}(F) \cong \pi_0H^{+}(\widehat{F}) \, \wr \, \zz \cong A \wr \zz \cong G.\]
Thus, $G \in \mathcal P$ and so $\mathcal G'\subset \mathcal P$.

\medskip

\textbf{2.}
Conversely, let us show that $\mathcal P\subset \mathcal G'$.

Let $\Sigma\in\mathfrak{F}$ be a striped surface presented in the form~\eqref{equ:striped_surface} with canonical foliation $F$ and such $\diam\Gamma(\Sigma) = k$.
We should prove that $\pi_0\Hplus$ has a finite presentation in the class $\mathcal G$, which means that $\pi_0\Hplus\in \mathcal G'$. %

If $k=0$, then $\Sigma$ is an admissible model strip with
\iffalse
\begin{align*}
\partial_{-}\Sigma &= \typeFinite{1}\times\{-1\},  &
\partial_{+}\Sigma &= \typeGen{\alpha}, \ \alpha \in \{ [0], [1],\ldots,\ntN, \ntNInv, \ntZ\}.
\end{align*}
\fi
$$
\partial_{-}\Sigma = \typeFinite{1}\times\{-1\},\quad
\partial_{+}\Sigma = \typeGen{\alpha}, \quad  \alpha \in \{ [0], [1],\ldots,\ntN, \ntNInv, \ntZ\}.
$$
Then it easily follows from Theorem~\ref{structure_theorem} that $\pi_0\Hplus \cong \zz \cong \{1\}\wr\zz$ if $\alpha = \ntZ$, and $\pi_0\Hplus \cong \{1\}$ otherwise.
In both cases $\pi_0\Hplus \in \mathcal G$.

Suppose that we have established our statement for all $k$ being less than some $\bar{k}>0$.
We should prove it for $k=\bar{k}$.
Let
\[
\Sigma \ = \ S \ \mathop{\cup}\limits_{\partial_{+}S} \ \Bigl(  \bigcup \limits_{\iind \in \Iind} \Sigma_{\iind} \Bigr) \ \in \ \mathfrak{F}
\]
be such that $\Gamma(\Sigma)$ has diameter $k$.
Then $\Gamma(\Sigma_{\iind})$ has diameter less than $k$, and so by inductive assumption $\pi_0 \HFdSi{\iind}\in\mathcal G$.
Moreover, according to Theorem~\ref{structure_theorem} we have that
\begin{enumerate}
\item [(i)]
if $\mathrm{image}(\eta)=0$, then $\pi_0\Hplus\cong \prod \limits_{\iind \in \Iind} \pi_0 \HFdSi{\iind} \in \mathcal G$,

\item [(ii)] if $\mathrm{image}(\eta) =k\zz$, then
$\pi_0\Hplus\cong  \left( \prod \limits_{\iind= 0}^{k-1} \pi_0  \HFdSi{\iind} \right) \wr \zz \in \mathcal G$.
\end{enumerate}
Thus $\mathcal P\subset \mathcal G'$, and so $\mathcal P = \mathcal G'$.
Theorem~\ref{structure_theorem} completed.

%\iffalse

%\bibliographystyle{amsplain}
%\bibliography{biblio}

\def\cprime{$'$} \def\cprime{$'$} \def\cprime{$'$} \def\cprime{$'$}
\providecommand{\bysame}{\leavevmode\hbox to3em{\hrulefill}\thinspace}
\providecommand{\MR}{\relax\ifhmode\unskip\space\fi MR }
\providecommand{\MRhref}[2]{%
  \href{http://www.ams.org/mathscinet-getitem?mr=#1}{#2}
}
\providecommand{\href}[2]{#2}

%\fi

\end{document}